\documentclass[a4paper,british]{amsart}

\usepackage{amscd}
\usepackage{amssymb} 
\usepackage[utf8]{inputenc}
\usepackage[arrow,curve,matrix]{xy}
\usepackage{times}
\usepackage{graphicx}
\usepackage{color}
\usepackage{comment}
\usepackage{enumitem}
\usepackage{varioref}
\usepackage{xstring}
\usepackage{hyperref}
\usepackage{xspace}
\usepackage{amsfonts}
\usepackage{amsmath}
\usepackage{latexsym}
\usepackage{mathrsfs}
\usepackage{amsthm}
\usepackage{verbatim}
\usepackage{bbm}

\usepackage{mathtools}
\usepackage[arrow,curve,matrix]{xy}
\allowdisplaybreaks

\usepackage{colortbl}
\usepackage{graphicx}
\usepackage[retainorgcmds]{IEEEtrantools}

\theoremstyle{plain}
\newtheorem{theorem}{Theorem}[section]
\newtheorem*{thma}{Theorem A}
\newtheorem*{thmb}{Theorem B}
\newtheorem*{thmc}{Theorem C}
\newtheorem*{thmd}{Theorem D}

\newtheorem{proposition}[theorem]{Proposition}
\newtheorem{corollary}[theorem]{Corollary}
\newtheorem{lemma}[theorem]{Lemma}
\newtheorem{conjecture}[theorem]{Conjecture}

\theoremstyle{definition}
\newtheorem{definition}[theorem]{Definition}
\newtheorem{example}[theorem]{Example}
\theoremstyle{remark}
\newtheorem{remark}[theorem]{Remark}

\theoremstyle{plain}
\newtheorem{thm}{Theorem}[section]

\numberwithin{equation}{thm}
\theoremstyle{plain}

\theoremstyle{plain}

\newtheorem{prop}[thm]{Proposition}
\newtheorem{proclaim-special}[thm]{\specialthmname}

\theoremstyle{remark}

\newtheorem*{claim*}{Claim}

\newtheorem{notation}[thm]{Notation}

\newtheorem{application-idea}[thm]{Idea of Application}

\DeclareMathOperator{\Bs}{Bs}

\DeclareMathOperator{\Exc}{Exc}

\DeclareMathOperator{\ord}{ord}

\DeclareMathOperator{\red}{red}
\DeclareMathOperator{\reg}{reg}

\DeclareMathOperator{\Sing}{Sing}
\DeclareMathOperator{\Sym}{S}

\newcommand{\wtilde}{\widetilde}

\newcommand{\sF}{\scr{F}}
\newcommand{\sG}{\scr{G}}

\newcommand{\sO}{\scr{O}}

\newcommand{\quotient}[2]{{\left.\raisebox{.2em}{$#1$}\middle/\raisebox{-.2em}{$#2$}\right.}}
\newcommand{\abs}[1]{ {\left| #1 \right| } }
\newcommand{\norm}[1] { \| #1 \| }

\DeclareMathOperator{\weight}{weight}

\DeclareFontFamily{OMS}{rsfs}{\skewchar\font'60}
\DeclareFontShape{OMS}{rsfs}{m}{n}{<-5>rsfs5 <5-7>rsfs7 <7->rsfs10 }{}
\DeclareSymbolFont{rsfs}{OMS}{rsfs}{m}{n}
\DeclareSymbolFontAlphabet{\scr}{rsfs}

\definecolor{linkred}{rgb}{0.7,0.2,0.2}
\definecolor{linkblue}{rgb}{0,0.2,0.6}

\setdescription{labelindent=\parindent, leftmargin=2\parindent}
\setitemize[1]{labelindent=\parindent, leftmargin=2\parindent}
\setenumerate[1]{labelindent=0cm, leftmargin=*, widest=iiii}

\numberwithin{figure}{section}

\usepackage[hyperpageref]{backref}

\usepackage{mathpazo}

\setlist[enumerate]{label=(\thetheorem.\arabic*), before={\setcounter{enumi}{\value{equation}}}, after={\setcounter{equation}{\value{enumi}}}}

\numberwithin{equation}{theorem}

\def\C{\mathbb C}
\def\B{\mathbb B}

\def\Q{\mathbb Q}

\def\D{\mathbb D}

\def\H{\mathfrak H}

\makeatletter
\let\saveqed\qed
\renewcommand\qed{%
   \ifmmode\displaymath@qed
   \else\saveqed
   \fi}
\makeatother

\makeatletter
\hypersetup{
  pdfauthor={\authors},
  pdftitle={\@title},
  pdfsubject={\@subjclass},
  pdfkeywords={\@keywords},
  pdfstartview={Fit},
  pdfpagelayout={TwoColumnRight},
  pdfpagemode={UseOutlines},
  bookmarks,
  colorlinks,
  linkcolor=linkblue,
  citecolor=linkred,
  urlcolor=linkred}
\makeatother

\begin{document}
\bibliographystyle{alpha}

\title{Hyperbolicity of singular spaces} 

\author{Beno\^it Cadorel}
\address{Beno\^it Cadorel \\Aix Marseille Univ\\ 
		CNRS, Centrale Marseille, I2M\\
		Marseille\\
		France} 
\email{\href{mailto:benoit.cadorel@univ-amu.fr}{benoit.cadorel@univ-amu.fr}}

\author{Erwan Rousseau}
\address{Erwan Rousseau \\ Institut Universitaire de France
	\& Aix Marseille Univ\\ 
		CNRS, Centrale Marseille, I2M\\
		Marseille\\
		France} 
\email{\href{mailto:erwan.rousseau@univ-amu.fr}{erwan.rousseau@univ-amu.fr}}

\author{Behrouz Taji}

\address{Behrouz Taji, University of Notre Dame, Department of Mathematics, 278 Hurley, Notre Dame, IN
46556, USA}
\email{\href{mailto:btaji@nd.edu}{btaji@nd.edu}}
\urladdr{\href{http://sites.nd.edu/b-taji/}{http://sites.nd.edu/b-taji/}}

\thanks{E. R. was partially supported by the ANR project \lq\lq FOLIAGE\rq\rq{}, ANR-16-CE40-0008. E.R. and B.T. thank the University of Sydney for the invitation at the School of Mathematics where part of this work was done.}

\date{\today}

\begin{abstract}
We study the hyperbolicity of singular quotients of bounded symmetric domains. We give effective criteria for such quotients to satisfy Green-Griffiths-Lang's conjectures in both analytic and algebraic settings. As an application, we show that Hilbert modular varieties, except for a few possible exceptions, satisfy all expected conjectures.
\end{abstract}

\maketitle
\tableofcontents

\section{Introduction}
As central objects in algebraic geometry, the geometry of quotients of bounded symmetric domains $\Omega/\Gamma$ has been the object of many works. Classical results state that there exist sufficiently small subgroups $\Gamma' \subset \Gamma$ such that $\Omega/\Gamma'$ has remarkable properties: it is of general type \cite{Mum77}, it is hyperbolic modulo the boundary (\cite{Na89}, \cite{Rou16}), all its subvarieties are of general type \cite{Bru}.

\bigskip

Nevertheless, it turns out that these properties should be true in most cases without having to take small subgroups $\Gamma' \subset \Gamma$. As an example, Tsuyumine \cite{Tsu85} has shown in the Hilbert modular case ($\Omega=\H^n$ and $\Gamma$ the Hilbert modular group) that, except finitely many cases, Hilbert modular varieties are of general type. The main difficulty, if one wants to avoid the step of taking small subgroups, is to deal with singularities. Indeed, it is well known that hyperbolicity properties may be completely lost in singular quotients (see for example Keum's singular ball quotient \cite{Keum08}).

\bigskip

The above mentioned results can be seen as illustrations of the expected following conjectures of Lang and Green-Griffiths (see \cite{Lan86} and \cite{GG80}).

Let $\Exc(X) \subset X$ denote the Zariski closure of the union of the images of all non-constant holomorphic maps $\C \to X$.

\begin{conjecture}\label{Lang}
Let $X$ be a complex projective manifold. Then $X$ is of general type if and only if $\Exc(X) \neq X$.
\end{conjecture}

In particular, if we denote $\Exc_{\text{alg}}(X)$ the Zariski closure of the union of non general type subvarieties then the conjecture implies $\Exc(X)=\Exc_{\text{alg}}(X).$
This conjecture is still largely open despite some very recent results for subvarieties in quotients of bounded domains \cite{BD18} or more generally in manifolds with negative holomorphic sectional curvature \cite{G18}.

The setting we consider in this article is the following: a quotient $X = \quotient{\Omega}{\Gamma}$ of a  bounded symmetric domain by an arithmetic lattice whose action is fixed-point free in codimension one.

Then (see \cite{mum2}) taking a finite index congruence subgroup $\Gamma'$, we can consider a compactification $\overline{X}=\overline{X'}/G$ obtained as a quotient of a smooth compactification $\overline{X'}$ by a finite group $G = \quotient{\Gamma}{\Gamma'}$ and denote $D = \overline{X} \setminus X$.

Denote by $p : \Omega \longrightarrow X$ the canonical projection. Let $\widetilde{X} \overset{\pi}{\longrightarrow} \overline{X}$ be a resolution with exceptional divisor $E_i$ corresponding to quotient singularities with local isotropy groups $G_i$. Let $\widetilde{\Delta} = \sum_i \left(1 - \frac{1}{\abs{G_i}} \right) E_i$, $E = \sum_i E_i$ and let $\widetilde{D} = \pi^\ast D$.

Let $h_{\mathrm{Berg}}$ be the Bergman metric on $\Omega$, which we normalize to have $\mathrm{Ric}(h_{\mathrm{Berg}}) = -h_{\mathrm{Berg}}$. Let $\gamma \in \mathbb Q_+^\ast$ such that the holomorphic sectional curvature of $h_{\mathrm{Berg}}$ is bounded from above by $- \gamma$.

The first result is a generalization of Nadel's theorem \cite{Na89} to this singular setting, both for the analytic and algebraic versions.

\begin{thma}\label{thmentire} Consider the $\mathbb Q$-line bundle
$$
L := \pi^\ast K_{\overline{X}} + \widetilde{D}  - \frac{1}{\gamma} \left( \widetilde{D} + \widetilde{\Delta} \right).
$$
Then $\Exc(\widetilde{X})$ and $\Exc_{\text{alg}}(\widetilde{X})$ are contained in $\mathbb B^+(L) \cup \widetilde{D} \cup E$, where $\mathbb{B}^+(L)$ denotes the augmented base locus.
In particular, if $L$ is big then the union of entire curves and subvarieties not of general type is not Zariski dense.
\end{thma}

Recall that the augmented base locus is defined as $$\mathbb{B}^+(L):=\bigcap_{m>0} \Bs(mL-A)$$ for any ample line bundle $A$.

In the case of the ball and the polydisc, we can prove stronger statements, namely the bigness of the cotangent bundle of subvarieties (which implies the bigness of the canonical bundle by \cite{CP15}).

\begin{thmb} \label{thmball} Let $\Omega = \mathbb B^n$ be the unit ball. Consider the $\mathbb Q$-line bundle 
$$
L = \pi^\ast K_{\overline{X}} + \widetilde{D} - (n+1) \left [\widetilde{\Delta} + \widetilde{D} \right].
$$ Then for any subvariety $V \subset \widetilde{X}$ such that $V \not\subset \mathbb{B}^+ (L) \cup \widetilde{D} \cup E$, any resolution of singularities $\widetilde{V}$ has big cotangent bundle.
\end{thmb}

\begin{thmc} When $\Omega$ is the polydisc, Theorem B holds with $(n+1)$ in the definition of $L$ replaced 
by $n$.
\end{thmc}

One of the key point in these three statements (and their proofs) is to consider the pairs $(X, \widetilde{\Delta}+\widetilde{D})$ as \emph{orbifolds} in the sense of Campana \cite{Cam04}. In section \ref{extension}, we recall basics on orbifolds and study extension properties in the orbifold category of the Bergman metric and symmetric differentials.

The proof of Theorems A and B is given in Section \ref{ball}. They are obtained by using negativity properties of the Bergman metric and its extension property as an orbifold metric described in section \ref{extension}, combined with the methods of \cite{cad16} and \cite{G18}. The common idea in both statements is to use sections of the line bundle $L$ to twist the Bergman metric and obtain in this way a singular metric on the compactification with the required negative curvature properties.

Theorem C is proved in section \ref{disc}. We use extension properties of orbifold symmetric differential forms explained in section \ref{extension} and properties of the natural codimension one holomorphic foliations living on these varieties. The idea is that in this setting, sections of the line bundle $L$ naturally induce symmetric differentials on the regular part of the quotient. Then, extension properties of these differential forms as orbifold symmetric differentials are used to construct global holomorphic symmetric differentials on the compactification. They define (multi)-foliations whose properties enable to derive the above statement.

\bigskip

As an application of the above results, we study in section \ref{HMV} the case of Hilbert modular varieties and obtain the following version of conjecture \ref{Lang} in this case.

\begin{thmd}
Let $n \geq 2$. Then, except finitely many possible exceptions, Hilbert modular varieties of dimension $n$ satisfy the following properties:
\begin{enumerate}
\item $\Exc(X) \neq X$.
\item there is a proper subvariety $Z$ such that all subvarieties not contained in $Z$ have big cotangent bundle, and in particular are of general type.
\end{enumerate}
\end{thmd}

The first part of this theorem was already obtained by a different method in \cite{RT} while the second part is a generalization of results of Tsuyumine \cite{Tsu86} who treated the case of codimension one subvarieties.

\section{Orbifolds}\label{extension}

Our aim in the current section is to establish extension results for the Bergman metric as an orbifold metric and
for orbifold differential forms which will play a key role in the proofs of our statements.

\subsection{Preliminaries on orbifolds}
 The notions we are about to introduce in the current section originated in the 
 works of Campana, cf.~\cite{Cam04}. For a more thorough account of these
 preliminary constructions, including background and applications, 
 the reader could also consult~\cite{MR2860268},~\cite{Taji16},~\cite{CKT16} and~\cite{GT16}.

\begin{definition}\label{def:pair}
We define an orbifold pair $(X,D)$ by a normal, algebraic variety $X$ and 
a divisor $D = \sum_{i=1}^l d_i \cdot D_i$, with $d_i = (1- \frac{1}{a_i})$, where 
$a_i \in \mathbb N^+ \cup \{ \infty \}$ and each $D_i$ is a prime divisor. 

\end{definition}

We follows the usual convention that, for $a_i = \infty$, we have $(1- \frac{1}{a_i}) = 1$.

The pair $(X,D)$ in Definition~\ref{def:pair} is sometimes referred to as a pair with standard coefficients
(or a classical or integral orbifold pair).

We say that the orbifold $(X,D)$ is \emph{smooth} if $X$ is smooth and the support of $D$ has normal crossing support.

With Definition~\ref{def:pair} at hand, given a pair $(X,D)$, one can naturally associate a notion of 
multiplicity to each irreducible component $D_i$ of $D$.

\begin{definition}
Let $(X,D)$ be an orbifold pair. We define the orbifold multiplicity $m_{D}(D_i)$ of each 
prime divisor $D_i$ as follows. 

$$
  m_D(D_i) := \left\{ 
    \begin{matrix}
      a_i & \text{, if $d_i\neq 1$ } \\
      \infty & \text{, if $d_i=1.$  }
    \end{matrix}
  \right.
  $$

\end{definition}

Our aim in now to introduce morphisms sensitive to the orbifold structure of $(X,D)$, but to do so 
we first need to define a notion for pullbacks of Weil divisors over normal varieties. 

\begin{definition}\label{def:pullback}
Let $f: Y\to X$ be a finite morphism of normal, algebraic varieties $X$ and $Y$. 
Let $D\subset X$ be a Weil divisor. We define 
$f^*(D)$ by the Zariski closure of the Weil divisor defined by $f^*(D|_{X_{\reg}})$.
\end{definition}

\medskip

We now turn to morphisms adapted to $(X,D)$. Such morphisms are guaranteed to 
exist whenever $X$ is smooth, thanks to Kawamata's covering constructions, 
cf.~\cite[Prop.~4.1.12]{Laz04-I}. 

\begin{definition}\label{def:adapted}
Let $(X,D)$ be an orbifold pair in Definition~\ref{def:pair}. We call a finite morphism 
$f: Y \to X$ of algebraic varieties $X$ and $Y$, strictly adapted (to $D$), if $Y$ is normal and 
$f^*(D_i) = a_i \cdot D'$, for some reduced divisor $D'$ in $Y$.

\end{definition}

We can also consider collections of charts that are adapted to the orbifold 
structure of a given pair $(X,D)$. 

\begin{definition}[Orbifold structures]\label{def:structure}
Given an orbifold pair $(X,D)$, let $\{ U_{\alpha}\}_{\alpha}$ be a Zariski open covering of $X$. 
We call a collection $\mathcal C_{\alpha} = \{ (U_{\alpha}, f_{\alpha}, X_{\alpha} )\}_{\alpha}$ 
of triples $(U_{\alpha}, f_{\alpha}, X_{\alpha})$
consisting of strictly adapted morphisms $f_{\alpha}: X_{\alpha} \to (U_{\alpha}, D|_{U_{\alpha}})$, 
an orbifold structure (or strictly adapted orbifold structure) associated to $(X,D)$.

\end{definition}

\medskip

\begin{definition}[Orbifold (pseudo-)metric]
Let $(X,D)$ be a smooth orbifold pair equipped with a smooth orbifold structure 
$\mathcal C_{\alpha} = \{ (U_{\alpha}, f_{\alpha}, X_{\alpha})  \}_{\alpha}$. An orbifold  (pseudo-)metric $\omega_D$ is a (pseudo-)metric on $X\setminus D_{\text{red}}$ such that $f_{\alpha}^*(\omega_D)$ extends as a (pseudo-)metric on $X_\alpha$.
\end{definition}

\begin{remark}
Let $(X, D =\sum_{i=1}^l (1-\frac{1}{a_i})\cdot D_i)$. Assume that $U_{\alpha}$ admits a coordinate system $(z_1, \ldots, z_n)$ such that $U_{\alpha}\cap D = \{ (z_1, \ldots, z_n) \; \big| \;  \prod_{i=1}^l z_i =0   \}$. If $\omega_D=i\sum_{j,k} \omega_{j,k} dz_j \wedge d\overline{z_k}$ then it extends as a an orbifold pseudo-metric if the function $$|z_j|^{2(1-1/a_j)}\omega_{j,j}$$ extends as a smooth function.
\end{remark}

\begin{notation}
Given a coherent sheaf $\sF$ on a normal algebraic variety $X$, by $\Sym ^{[\bullet]}\sF$
we denote the reflexive hull $(\Sym^{\bullet}\sF)^{\vee\vee}$ of symmetric powers 
$\Sym^{\bullet}\sF$ of $\sF$. Furthermore, for a morphism of normal 
algebraic varieties $f: Y \to X$, we set $f^{[*]}\sF$ to denote $(f^*\sF)^{\vee \vee}$.

\end{notation}

We can now define a notion of sheaves of differential forms adapted to the orbifold 
structure of $(X,D)$. First, we need to fix some notations. 

\begin{notation}\label{not:orbi}
Let $\mathcal C_{\alpha}$ be an orbifold structure for the orbifold pair $(X,D)$.
Let $\{ D_{X_{\alpha}}^{ij} \}_{i, j}$ be the collection of prime divisors in $X_{\alpha}$ verifying the 
equality $\bigl(f^*(D_i)\bigr)_{\mathrm red}= \sum_j  D^{ij}_{X_{\alpha}}$. 
Let $X_{\alpha}^\circ$ denote the smooth locus of 
$(X_{\alpha}, f_{\alpha}^*(D))$ and $U_{\alpha}^\circ$ an open subset of 
the smooth locus of $(U_{\alpha}, D)$ such that $f_{\alpha}: X_{\alpha}^\circ\to U_{\alpha}^\circ$ is surjective.
\end{notation}

\begin{definition}\label{def:OrbiCot}
In the setting of Notation~\ref{not:orbi}, we define the 
orbifold cotangent sheaf $\Omega_{\mathcal C_{\alpha}}^{[1]}$ by the collection 
of reflexive, $\sO_{X_{\alpha}}$-module, coherent sheaves $\Omega^{[1]}_{(X_{\alpha}, f_{\alpha}, D)}$ 
uniquely determined as the coherent extension of the 
kernel of the sheaf morphism 

$$
f_{\alpha}^*\Omega^1_{U_{\alpha}^\circ} \big(  \log ( \ulcorner D   |_{U_{\alpha}^\circ} \urcorner) \big) 
   \longrightarrow \bigoplus \sO_{D_{X_{\alpha}^\circ}^{ij}},
$$
naturally defined by the residue map. 
\end{definition}

\begin{remark}
This is an instance of the notion of orbifold sheaves adapted to a pair $(X,D)$.
These are $G_{\alpha}$-linearized, coherent sheaves on each $X_{\alpha}$
satisfying natural compatibility conditions, cf.~\cite[Sect.~2.6]{GT16}.
\end{remark}

\begin{example}\label{ex:OrbiCot}
Let $(X, D =\sum_{i=1}^l (1-\frac{1}{a_i})\cdot D_i)$ be an orbifold pair and $f_{\alpha}: X_{\alpha}\to U_{\alpha}$ 
 a single, strictly adapted chart in Definition~\ref{def:structure}.
Assume that the branch locus of $f_{\alpha}$ is divisorial and has simple normal crossing support. 
With this assumption, it follows that $X_{\alpha}$ is smooth. 
Furthermore, let us assume that $U_{\alpha}$ admits a coordinate system $(z_1, \ldots, z_n)$
such that $U_{\alpha}\cap D = \{ (z_1, \ldots, z_n) \; \big| \;  \prod_{i=1}^l z_i =0   \}$.
Assume that $w_1, \ldots , w_n$ is a coordinate system on $X_{\alpha}$ such that 

\begin{equation} \label{expradaptedcover}
f_{\alpha}: (w_1, \ldots, w_n) \mapsto (w_1^{a_1} , \ldots, w_l^{a_l}, w_{l+1}^{b_1}, \ldots, w_{l+k}^{b_k},
w_{l+k+1}, \ldots, w_n),
\end{equation}
for some $b_1, \ldots, b_k \in \mathbb N$.
Then, by definition, the sheaf $\Omega^1_{(X_{\alpha}, f_{\alpha}, D )}$
is the $\sO_{X_{\alpha}}$-module generated by 
$$
\langle    dw_1  , \ldots, dw_l,  (w^{b_1-1}_{l+1})\cdot dw_{l+1}, \ldots, (w^{b_k-1}_{l+k})\cdot dw_{l+k}, 
dw_{l+k+1}, \ldots, dw_n \rangle. 
$$

\end{example}

\medskip

\begin{definition}
Give any orbifold pair $(X,D)$ together with an orbifold structure 
$\mathcal C_{\alpha} =  \{ (U_{\alpha}, f_{\alpha}, X_{\alpha}) \}_{\alpha}$, 
we define the sheaf of symmetric orbifold differential forms $\Sym^{[\bullet]} \Omega^{[1]}_{\mathcal C_{\alpha}}$
by the collection of reflexive sheaves $\{ \Sym^{[\bullet]} \Omega^{[1]}_{(X_{\alpha}, f_{\alpha}, D)}  \}_{\alpha}$
on each $X_{\alpha}$.

\end{definition}

\medskip

\begin{notation}
Let $(X,D)$ be an orbifold pair. We shall denote by $\Sym^{[\bullet]}\Omega^{[1]}_X(*  D_{\red})$
the sheaf of symmetric rational differential forms with poles of arbitrary order along $D_{\red}$,
which is defined by:

$$
\lim_{\substack{\longrightarrow\\ m}} \Bigl(S^{[\bullet]} \Omega_X^{[1]} \otimes \sO_X(m\cdot D_{\red}) \Big)^{\vee\vee}.
$$

\end{notation}

Sometimes it is more convenient to work with a notion of symmetric differential forms, 
adapted to $(X,D)$, as a coherent sheaf on $X$ instead of $X_{\alpha}$. 
This is the purpose of the following definition.

\begin{definition}
\label{def:CDiff}
Let $(X,D)$ be an orbifold pair equipped with an orbifold structure 
$\mathcal C_{\alpha} = \{ (U_{\alpha}, f_{\alpha}, X_{\alpha})  \}_{\alpha}$.
We define the sheaf of $\mathcal C$-differential forms $\Sym^{[\bullet]}_{\mathcal C} \Omega^{[1]}_X\log (D)$
to be the reflexive, coherent subsheaf of $\Sym^{[\bullet]} \Omega^{[1]}_X(*D_{\red})$, defined, at the level of 
presheaves, by the following property:

\begin{IEEEeqnarray}{rCl}
\sigma \in \Gamma (U_{\alpha}, \Sym^{[\bullet]}_{\mathcal C} \Omega^{[1]}_X\log (D)) \; \; \Longleftrightarrow \;\;
   \sigma \in \Gamma(U_{\alpha}, S^{[\bullet]} \Omega^{[1]}_X(* D_{\red}) ) \; \; \text{and} \nonumber \\
    f^{[*]}_{\alpha}(\sigma) \in \Gamma (X_{\alpha}, \Sym^{[\bullet]} \Omega^{[1]}_{(X_{\alpha, f_{\alpha}, D})} ).\nonumber
\end{IEEEeqnarray}

\end{definition}

\begin{remark}
One can easily check that the notion of $\mathcal C$-differential forms in Definition~\ref{def:CDiff}
is independent of the choice of the orbifold structure $\mathcal C_{\alpha}$.
\end{remark}

\begin{remark}
Local calculations show that over $X_{\reg}$ the sheaf $\Sym_{\mathcal C}^{[\bullet]}\Omega^{[1]}_{X}\log(D)$ 
is locally free. More precisely, for every $x\in X_{\reg}$, there
exists an open neighbourhood $W_x\subset X_{\reg}$ with $W_x\cap D = \{ (z_1, \ldots, z_n) \; \big| \; \prod_{i=1}^l z_i =0  \}$ 
such that $\Sym^N_{\mathcal C}\Omega^1_{X_{\reg}}\log (D|_{\reg})$ is the $\sO_W$-module 
freely generated by 
$$
\langle \frac{dz_1^{m_1}}{z_1^{\lfloor m_1\cdot d_1 \rfloor}}, \ldots, \frac{dz_l^{m_l}}{z_l^{\lfloor m_l\cdot d_l \rfloor}},
     dz_{l+1}^{m_{l+1}}, \ldots, dz_n^{m_n}  \rangle, \;  \; \; \sum_{i=1}^n m_i = N,
$$
where $d_i= (1-\frac{1}{a_i})$.
\end{remark}

\subsection{Extension theorems}
We are now ready to prove extension results for metrics and differential forms.
In the case of differential forms, it can be interpreted as an
orbifold version of~\cite[Prop.~3.1]{GKK08}.

The setting of this section is the following. Let $U\subseteq \mathbb C^n$ be a normal, algebraic subset. Assume 
that there is a smooth algebraic subset $V\subseteq \mathbb C^n$
and a finite group $G$ acting freely in codimension one such that $U\cong V/G$. 
Let $\pi: \wtilde U \to U$ be a strong log-resolution of $U$ with $E\subseteq \Exc(\pi)$ being the 
maximal reduced exceptional divisor and $D=(1-\frac{1}{| G |})\cdot E$.

Consider an invariant metric $\omega$ on $V$ which induces a metric $\omega_D$ on $\wtilde U \setminus E$.

\begin{proposition}
\label{prop:metricextension}
$\omega_D$ extends as an orbifold pseudo-metric on $(\wtilde U, D)$.
\end{proposition}

\begin{proof}
Let $f: V\to U$ denote the finite map encoding the isomorphism $U \cong V/G$
and set $\wtilde V$ to be the normalization of the fibre product $V\times_U \wtilde U$ with 
the resulting commutative diagram:

  $$
  \xymatrix{
    \wtilde V \ar[rrr]^{\wtilde f, \; \text{finite}} \ar[d]_{\wtilde \pi} &&& \wtilde U \ar[d]^{\pi} \\
      V \ar[rrr]^{f \;=\; \cdot/G}     &&& U
  }
  $$

 Note that $\wtilde f$ is a composition of the finite map $V \times_U \wtilde U\to \wtilde U$, which is 
of degree $|G|$, and its normalization. Therefore $\deg(\wtilde f) = | G |$.
Set $\{E_i$\} to be the set of irreducible components of $E$ and let 
$$
E' = \sum (1- \frac{1}{a_i})\cdot E_i
$$
be the divisor with respect to which $\wtilde f : \wtilde V \to \wtilde U$ is strictly adapted 
to the orbifold pair $(\wtilde U, E')$
(cf.~Definition~\ref{def:adapted}).

From Theorem 2.23 in \cite{Kol07}, we know that $\wtilde V$ has quotient singularities. Therefore local uniformizations for $\wtilde V$ give a smooth orbifold structure for $(\wtilde V, E')$. As $\omega$ pulls-back to a pseudo-metric on these local uniformizations, $\omega_D$ extends as an orbifold pseudo-metric for $(\wtilde V, E')$ and therefore on $(\wtilde U, D)$, since by construction, we have that for each $i$ the inequality $a_i \leq | G |$ holds.

\end{proof}

Now, we will prove an extension property for symmetric differential forms.

\begin{proposition}
\label{prop:extension}
For every $m\in \mathbb N$, the
coherent sheaf 
\begin{equation}\label{def:sheaf}
\sG: = \pi_*\Big( \Sym^m_{\mathcal C} \Omega^1_{\wtilde U}\log \big( (1-\frac{1}{| G |})\cdot E  \big)   \Big)
\end{equation}
is reflexive. 

\end{proposition}

\begin{proof}
The proof uses the same construction as in the previous proposition, so we follow the same notations using 
the following commutative diagram:

  $$
  \xymatrix{
    \wtilde V \ar[rrr]^{\wtilde f, \; \text{finite}} \ar[d]_{\wtilde \pi} &&& \wtilde U \ar[d]^{\pi} \\
      V \ar[rrr]^{f \;=\; \cdot/G}     &&& U
  }
  $$

As the problem is local, 
to prove the reflexivity of the sheaf $\sG$ (\ref{def:sheaf}),
it suffices to show that the naturally defined map 

\begin{equation}\label{eq:surject}
\Gamma\big(\wtilde U, \Sym^m_{\mathcal C} \Omega^1_{\wtilde U} \log \Big( ( 1-\frac{1}{| G |}  ) \cdot E  \Big) \big)
       \longrightarrow    \Gamma \big(\wtilde U\backslash \Exc(\pi),  \Sym^m \Omega^1_{\wtilde U} \big)
\end{equation}
 is surjective. 
 
 To this end, let $\sigma \in \Gamma(U, \Sym^{[m]} \Omega^{[1]}_U)$ 
 be the section defined by $\wtilde \sigma \in \Gamma\big( \wtilde U\backslash \Exc(\pi) ,\Sym^m \Omega^1_{\wtilde U}\big)$.
 Let $\sigma_V = f^{[*]}\sigma \in \Gamma(V, \Sym^m \Omega^1_V)$.
  As $\sigma_V$ is regular, we have 
  $$
   \wtilde \pi^* f^{[*]} (\sigma) =  \wtilde \pi^*(\sigma_V)\in \Gamma(\wtilde V ,\Sym^{[m]}\Omega^{[1]}_{\wtilde V}).
  $$
  Viewing $\wtilde \sigma$ as an element of $\Gamma\big(\wtilde U, \Sym^m \Omega^1_{\wtilde U}(* E)\big)$, 
  and thanks to the commutativity of the diagram above, we have
  
  \begin{equation}\label{eq:include}
   \wtilde f^*(\sigma)
   \in \Gamma(\wtilde V , \Sym^{[m]}\Omega^{[1]}_{\wtilde V}).
  \end{equation} 

Set as above $\{E_i$\} to be the set of irreducible components of $E$ and let 
$$
E' = \sum (1- \frac{1}{a_i})\cdot E_i
$$
be the divisor with respect to which $\wtilde f : \wtilde V \to \wtilde U$ is strictly adapted 
to the orbifold pair $(\wtilde U, E')$.
Note that, by construction, we have:

\begin{enumerate}
\item \label{item:1} For each $i$, the inequality $a_i \leq | G |$ holds. 
\item $\Omega^{[1]}_{\wtilde V} = \Omega^{[1]}_{(\wtilde V, \wtilde f, E')}$ (cf.~Example~\ref{ex:OrbiCot}).
\end{enumerate}

Therefore, thanks to (\ref{eq:include}) we have, by the definition of symmetric $\mathcal C$-differential forms, 
that $\sigma \in \Gamma\big(\wtilde U, \Sym^m_{\mathcal C} \Omega^1_{\wtilde U} \log (E')\big)$. In 
particular, it follows that  
the map

$$
\Gamma\big(\wtilde U,  \Sym^m_{\mathcal C} \Omega^1_{\wtilde U}\log (E')   \big)     \longrightarrow  
    \Gamma \big( \wtilde U\backslash \Exc(\pi), \Sym^m\Omega^1_{\wtilde U}    \big)$$

is a surjection.

On the other hand, thanks to the inequality $a_i \leq |G|$ in~\ref{item:1}, we have

$$
\Sym^m_{\mathcal C} \Omega^1_{\wtilde U}\log (E') \subseteq 
        \Sym^m_{\mathcal C} \Omega^1_{\wtilde U} \log\big(  (1- \frac{1}{|G|})\cdot E \big).
$$
 The surjectivity of the map~(\ref{eq:surject}) now follows. 
\end{proof}

We conclude this section by pointing out that
one can easily verify that these extension results hold for any 
resolution of $U$.

\section{Hyperbolicity and singular quotients}\label{ball}
This section will be devoted to the proof of Theorems A and B. The main ingredient in both proofs is the use of singular metrics built from the Bergman metric.

\subsection{Singular metrics}
In the following, we will consider a quotient $X = \quotient{\Omega}{\Gamma}$ of a bounded symmetric domain by an arithmetic lattice whose action is fixed-point free in codimension one, with local isotropy groups $G_i$ such that $|G_i|=m_i$.
\medskip

We will deal with a particular compactification $\overline{X}$ of $X$, constructed as follows. Since $\Gamma$ is arithmetic, we can take a congruence subgroup $\Gamma'$ which is \emph{neat}, by \cite{borel69}. Thus, we can use \cite{mum2} to construct a toroidal compactification $\overline{X'}$ of the smooth quotient $X' = \quotient{\Omega}{\Gamma'}$. The action of the finite quotient group $G = \quotient{\Gamma}{\Gamma'}$ on $X'$ extends to $\overline{X'}$. We then let $\overline{X} = \quotient{ \overline{X'}}{G}$, and  $D = \overline{X} \setminus X$. We also let $p : \overline{X'} \longrightarrow \overline{X}$ be the canonical projection. We note that $\overline{X}$ has normal and $\mathbb Q$-factorial, cf.~\cite[Lemma 5.16]{KM08}.
In particular the divisor $D$ is $\mathbb Q$-Cartier. 
\medskip

We choose a desingularization $\widetilde{X} \overset{\pi}{\longrightarrow} \overline{X}$ such that $(\Exc(\pi) + D)$ has simple normal crossing
support. Let $\widetilde{D} = \pi^\ast D$. Denote by $E_i$ the exceptional divisors supported over $\Sing(X)$. We let $\widetilde{\Delta} = \sum_i \left(1 - \frac{1}{m_i} \right) E_i$, and $E = \sum_i E_i$.
\medskip

We summarize this setting in the following diagram.

  $$
  \xymatrix{
     &&& \widetilde{X} \ar[d]^(0.35){\pi} \\
      \overline{X'} \ar[rrr]^{p}     &&& \quotient{\overline{X'}}{G} = \overline{X}
   } $$

The Bergman metric induces natural singular metrics on each one of the three varieties that appear above. Let us give more precise notations for these metrics.

Let $U = \overline{X} \setminus \left( \Sing(X) \cup D \right)$, $\widetilde{U} = \pi^{-1}(U)$ and $U' = p^{-1}(U)$. Then $U$, $\quotient{U'}{G}$ and $\widetilde{U}$ are all naturally biholomorphic. The Bergman metric $h_{\mathrm{Berg}}$ is well defined on $U'$, and is invariant under the action of $G$. Thus, $h_{\mathrm{Berg}}$ induces a metric

$$\norm{\cdot}_{\pi^\ast h_{\mathrm{Berg}}}$$
on the tangent bundle $T_{\widetilde{U}}$, and hence defines a metric $\norm{\cdot}_{ \pi^\ast \det h^\ast_{\mathrm{Berg}}}$ on the canonical bundle $K_{\widetilde{U}}$.
\medskip

We first state a lemma showing that the norm of the sections of powers of $\mathcal O_{\widetilde{X}} \left( \pi^\ast K_{\overline{X}} + \widetilde{D} \right)$ has logarithmic growth near the support of $\widetilde{D}$. This follows from the simple observation that, in our situation, the $\mathbb Q$-Cartier divisors $\pi^{\ast} \left( K_{\overline{X}} + D \right)$  and $K_{\overline{X'}} + D'$ are identified. Thus, it suffices to show the required estimate for sections of powers of $\mathcal O_{\overline{X'}} (K_{\overline{X'}} + D')$. But this latter estimate is a classical result of Mumford \cite{Mum77}.   

\begin{lemma} \label{lemgrowth} Let $m \in \mathbb N$ be sufficiently divisible, and let 
$$s \in H^0 \left( \widetilde{X}, \mathcal O_{\widetilde{X}}(m (\pi^\ast K_{\overline{X}} + \widetilde{D})) \right).$$ 
Let $\norm{ \cdot }$  denote the norm $\norm{\cdot}_{\pi^\ast \det h_{\mathrm{Berg}}^\ast {}^{\otimes m}}$ on $K_{\widetilde{U}}^{\otimes m}$. Then $\norm{s|_{\widetilde{U}}}^2$ is locally bounded near each point of $E \setminus \widetilde{D}$, and has at most logarithmic growth near the simple normal crossing divisor $\widetilde{D}$, in the sense of \cite{Mum77}. This means that if $w$ is a local equation for $\widetilde{D}$ on some small open set, we have
$$
\norm{s|_{\widetilde{U}}}^2 \leq C \left| \log |w| \right|^k,
$$
for some $C > 0$ and some $k > 1$.
\end{lemma}
\begin{proof}

Let $\overline{X}_{\mathrm{reg}}$ be the smooth locus of $\overline{X}$. By hypothesis, the group $G$ does not fix any subvariety of codimension $1$ inside $X'$. Thus, any codimension 1 component of the fixed locus of $G$ must be a component of $D'$. This implies that the map $p|_{p^{-1} (\overline{X}_{\mathrm{reg}})}$ is a ramified cover, which can ramify only along $D'\cap  p^{-1}(\overline{X}_{\mathrm{reg}})$. Consequently $p^\ast \pi_\ast s$ is a well-defined section of $m( K_{\overline{X'}} + D')$ on $p^{-1}(\overline{X}_{\mathrm{reg}})$. The map $p$ is finite, so the complement of this last set has codimension higher than $2$ in $\overline{X'}$. This shows that $p^\ast \pi_\ast s$ extends as a section $s'$ of $m( K_{\overline{X'}} + D')$ on $\overline{X'}$.

\medskip

The lattice $\Gamma'$ is neat, so by \cite{Mum77} the Bergman metric on $m(K_{\overline{X'}} + D')$ has at most logarithmic growth near $D'$. This shows that, for some $k > 0$, we have $\norm{s'} \leq C |\, \log |w'| \,| ^k$. Here, $w'$ is a local equation for $D'$, and $\norm{\cdot}$ is the norm induced by the Bergman metric on $U'$. 

\medskip

Since $s' = p^\ast \pi_\ast s$, we have $\norm{s} = \norm{s'} \circ p^{-1} \circ \pi$ on $\widetilde{U}$, where the norms $\norm{\cdot}$ are still induced by the Bergman metric on the adequate powers of $K_{\widetilde{X}} + \widetilde{D}$ and $K_{\overline{X'}} + D'$. Now note that on some suitable small open subset on $\overline{X}$, we can find a local equation $w'$ for $D'$ such that $\left| \log |w'| \right| \circ p^{-1} \leq C_1 \,   \left| \log \sum_i |w_i|  \, \right| + O(1)$, for some $C_1 > 0$, where $(w_i)_i$ is a local set of generators for the ideal sheaf of $D$. The previous assertion follows easily from the equality of ideals $\sqrt{\sum_i \left( p^\ast w_i \right)} = \left( w' \right)$, valid in the local ring of any point of $D'$.

Besides, on $\widetilde{X}$, we have $ \log \sum_i |w_i| \,  \circ \pi = C_2 \log |\widetilde{w}| + u$, where $\widetilde{w}$ is a local equation for $\widetilde{D}$, and $u$ is some locally bounded function. Putting everything together, we obtain the following inequality:
$$
\norm{ \left. s \right|_{\widetilde{U}} } \leq C_3 \abs{ \log \abs{\widetilde{w}}}^k + O(1).
$$
This gives the result, since the right-hand side has the required logarithmic grow near $\widetilde{D}$, and is locally bounded away from $\widetilde{D}$.
\end{proof}

Next, we turn to the control of the singularities of the metric $\pi^\ast h_{\mathrm{Berg}}$ near $\widetilde{D}$.

\begin{lemma} The metric $\pi^\ast h_{\mathrm{Berg}}$ has at most mixed cone and cusp singularities with respect to the divisor $\widetilde{\Delta} + \widetilde{D}$. This means that on any small polydisk $\mathbb D^n$ with coordinates $(z_1, \ldots , z_n)$ such that $E_i \cap \mathbb D^n = \left\{ z_i = 0 \right\}$ for $i = 1, \ldots, l$ and $\widetilde{D} \cap \mathbb D^n = \left\{ z_{l+r} \cdot \ldots \cdot z_{n} = 0 \right\}$, we have
\begin{equation} \label{mixedconicalcusp}
\norm{v}^2_{\pi^\ast h_{\mathrm{Berg}}} \leq C \left[ \sum_{i = 1}^{l} \frac{|v_i|^2}{|z_i|^{2(1 - \frac{1}{m_i})}} + \sum_{i=l+1}^{l+r-1} |v_i|^2 + \sum_{i = l + r }^{n} \frac{|v_i|^2}{|z_i|^2 \left| \log |z_i| \right|^2} \right].
\end{equation}
\end{lemma}

\begin{proof} 
Let $f_\alpha : X_\alpha \longrightarrow U_\alpha$ be a strictly adapted chart for the smooth orbifold pair $(X, \widetilde{\Delta})$. We can always choose $f_\alpha$ so that $f_\alpha^\ast (\widetilde{\Delta} + \widetilde{D})$ has simple normal crossing support. Let $\widetilde{D}_\alpha$ be the reduced divisor associated to $f^\ast_\alpha \widetilde{D}$. Let $T_{(X_\alpha, f_\alpha, \widetilde{\Delta})} = \left( \Omega^1_{(X_\alpha, f_\alpha, \widetilde{\Delta})} \right)^{\vee}$ be the orbifold tangent bundle on $X_\alpha$. Assuming that $f_\alpha$ has the expression \eqref{expradaptedcover} (with $a_i = m_i$), then $T_{(X_\alpha, f_\alpha, \widetilde{\Delta})}$ admits the local frame:
\begin{equation} \label{exprframetangent}
\langle    \frac{\partial}{\partial w_1}  , \ldots, \frac{\partial}{\partial w_l},  (w^{1 - b_1}_{l+1})\cdot \frac{\partial}{\partial w_{l+1}}, \ldots, (w^{1 - b_k}_{l+k})\cdot \frac{\partial}{\partial w_{l+k}}, 
\frac{\partial}{\partial w_{l+k+1}}, \ldots, \frac{\partial}{\partial w_n} \rangle. 
\end{equation}

To prove the result, it suffices to show that $X$ is covered by such charts $(U_\alpha, f_\alpha, X_\alpha)$, for which the metric $f_\alpha^\ast \pi^\ast h_{\mathrm{Berg}}$ on $T_{(X_\alpha, f_\alpha, \widetilde{\Delta})}$ has Poincaré growth with respect to $\widetilde{D}_\alpha$. We can choose $f_\alpha$ so that $w_{l+k+s} \cdot \ldots \cdot w_{n} = 0$ is a local equation for $\widetilde{D}_\alpha$, with $s \geq 1$. What we need to check is that for any local section $v$ of $T_{(X_\alpha, f_\alpha, \widetilde{\Delta})}$ with coordinates $v_i$ in the frame \eqref{exprframetangent}, we have 
$$
\norm{v}^2 \leq C \left[ \sum_{i =1}^{l+ k + s - 1} |v_i|^2 + \sum_{i=l+k+s}^{n}  \frac{|v_i|^2}{|w_i|^2 \left| \log |w_i| \right|^2} \right].
$$
where $\norm{\cdot}$ is the norm given by the metric $f_\alpha^\ast \pi^\ast h_{\mathrm{Berg}}$.

By Proposition \ref{prop:metricextension}, the metric $f^\ast_\alpha \pi^\ast h_{\mathrm{Berg}}$ is locally bounded on $T_{(X_\alpha, f_\alpha, \widetilde{\Delta})}|_{X_\alpha \setminus \widetilde{D}_\alpha}$. Moreover, on its non-singular locus, this metric has negative sectional holomorphic curvature and non-positive bisectional curvature. Then the argument of \cite[Proposition 2.4]{BC18} can be readily adapted to the orbifold case to show that the metric $f^\ast_\alpha \pi^\ast h_{\mathrm{Berg}}$ on $T_{(X^\alpha, f_\alpha, \widetilde{\Delta})}$ has Poincaré growth near $\widetilde{D_\alpha}$. Actually, the only additional fact to check is that $w_{l+j}^{1 - b_{j}}\frac{\partial}{\partial w_{l+j}}$ has uniformly bounded norm for $j = 1, \ldots, k$. As in the non-orbifold case, this follows from the Ahlfors-Schwarz lemma, using the fact that outside $\widetilde{D}_\alpha$, the metric $f_\alpha^\ast \pi^\ast  h_\mathrm{Berg}$ is locally bounded and has negative holomorphic sectional curvature.
\end{proof}

\begin{remark}
Another way to prove the last two results would be to check that the Bergman metric defines a singular Kähler-Einstein metric on the normal variety $\overline{X}$, belonging to the big class $c_1(K_{\overline{X}} + D)$. Then, \cite[Proposition 2, Theorem 3]{GP16} permits to conclude that the induced metric on $\widetilde{X}$ has mixed cone and cusp singularities.

This also implies that near a point of $\widetilde{D}$, the Bergman metric is equivalent to the Poincaré metric. We can then use this fact to show that any section of $m \pi^\ast (K_{\overline{X}}  + D)$ has bounded norm for $(\det h_{\mathrm{Berg}}^\ast ){}^{\otimes m}$.
\end{remark}

The following lemma is the last step before the proofs of Theorem A, B and C. It explains how to define a singular metric on $T_{\widetilde{X}}$ with suitable negative curvature properties, using sections of an appropriate adjoint line bundle.

\begin{lemma} \label{lemcritere} Let $- B \leq 0$ be an upper bound for the bisectional curvature of the Bergman metric on $\Omega$. Assume that the line bundle
$$
L = \pi^\ast  K_{\overline{X}} + \widetilde{D} - \frac{1}{A} \left[ \widetilde{D} + \widetilde{\Delta} \right]
$$
is big on $\widetilde{X}$ for some constant $A > 0$. Then for any $x \in \widetilde{X} \setminus \left( \mathbb B^+(L) \cup \widetilde{D} \cup E \right)$, there exists a singular hermitian metric $\widetilde{h}$ on $T_{\widetilde{X}}$, such that
\begin{enumerate}
\item $\widetilde{h}$ is smooth and non-degenerate near $x$;
\item on its smooth locus, $\widetilde{h}$ has bisectional curvature bounded from above by $- B + A$, and holomorphic sectional curvature bounded from above by $- \gamma + A$;
\item $\widetilde{h}$ is locally bounded everywhere on $\widetilde{X}$.
\end{enumerate}
\end{lemma}

\begin{proof}
Bigness is an open property, so there exist rational numbers $\beta_i > 1 - \frac{1}{m_i}$ and $\lambda > 1$ such that 

$$
L' = \pi^\ast K_{\overline{X}} + \widetilde{D} - \frac{1}{A} \left[ \lambda \widetilde{D} + \sum_i \beta_i E_i \right]
$$
is still big. Moreover, since $x \not\in \mathbb B^+ (L)$, we can choose the $\beta_i$ and $\lambda$ so that $x \not\in \mathrm{Bs}(L')$. Consequently, we can find a section $s \in H^0(\widetilde{X}, L'^{\otimes m})$ such that $s(x) \neq 0$, for some $m \geq 1$ high enough. Consider the metric 

\begin{equation} \label{defmetric}
\widetilde{h} = \norm{s}^{\frac{2 A}{m}}\cdot \pi^\ast h_{\mathrm{Berg}},
\end{equation} 
where $\norm{\cdot}$ is the norm induced by the Bergman metric on $m ( \pi^\ast K_{\overline{X}} + \widetilde{D} )$. We will show that $\widetilde{h}$ has all the required properties.
\medskip

We first compute the curvature of $\widetilde{h}$ on its smooth locus. Using our normalization assumptions, we find
\begin{align*}
i \Theta(\widetilde{h}) & = i \frac{A}{m} \Theta \left( \det {h^\ast_{\mathrm{Berg}}}^{\otimes m} \right) \otimes \mathbb{I} + i \Theta(h_{\mathrm{Berg}}) \\
	& = - A \pi^\ast \mathrm{Ric} (h_{\mathrm{Berg}}) \otimes \mathbb I+ i \Theta(h_{\mathrm{Berg}}) \\
	& =  A \pi^\ast \omega_{h_{\mathrm{Berg}}}  \otimes \mathbb I+ i \Theta(h_{\mathrm{Berg}}).
\end{align*}
and the required bounds on the curvatures are then given by a simple computation (see for example \cite{cad16}). 
\medskip

Now, let us prove that $\widetilde{h}$ is locally bounded everywhere. We see $s$ as a section of $m( \pi^\ast K_{\overline{X}} + \widetilde{D})$. By construction, it vanishes at order at least $\frac{m \lambda}{A}$ along $\widetilde{D}$. Thus, by Lemma \ref{lemgrowth}, the function $\norm{s}^{\frac{2A}{m}}$ vanishes on $\widetilde{D}$ at any order strictly smaller than $\frac{2 A}{m} \cdot \frac{m\lambda}{A} = 2 \lambda$. Since $\lambda > 1$, $\norm{s}^{\frac{2 A}{m}}$ vanishes at order $2$ along $\widetilde{D}$.

Besides, $s$ vanishes at order $\frac{m \beta_i}{A}$ along each $E_i$. So, by Lemma \ref{lemgrowth} again, $\norm{s}^{\frac{2A}{m}}$ vanishes at order $\frac{2A}{m} \cdot \frac{m \beta_i}{A} = 2 \beta_i > 2 \left( 1 - \frac{1}{m_i} \right)$ along any $E_i$. 

If we combine these two facts with \eqref{mixedconicalcusp}, and if we recall that $\widetilde{h}$ has the expression \eqref{defmetric}, we see that $\widetilde{h}$ is locally bounded everywhere.
\end{proof}

\subsection{Proofs}

Now, Lemma \ref{lemcritere} can be used to prove our criteria for complex hyperbolicity.

\begin{proof}[Proof of Theorem A]

By contradiction, assume there exists an entire curve  $\mathbb C \overset{j}{\longrightarrow} \widetilde{X}$ not included in $\mathbb B^+(L) \cup \widetilde{D} \cup E$. By our hypothesis, and since bigness is an open condition, we can pick $A <  \gamma$ in Lemma \ref{lemcritere}. This implies that there exists a singular metric $\widetilde{h}$ on $T_{\widetilde{X}}$, non-degenerate on $j(\mathbb C)$, locally bounded everywhere, and with holomorphic sectional curvature bounded from above by $- \gamma + A < 0$. 

Thus, $j^\ast \widetilde{h}$ defines a smooth metric on $\mathbb C$ outside a discrete set of points. This metric has negative curvature, bounded away from zero, and is locally bounded everywhere. We can apply the usual extension theorem for plurisubharmonic functions, to obtain a singular metric on $\mathbb C$, with negative curvature in the sense of currents, bounded away from zero. This is absurd by Ahlfors-Schwarz lemma.

Now, let $V \longrightarrow \widetilde{X}$, with $V \not\subset \mathbb B^+(L) \cup D \cup E$, and let $\widetilde{V}$ be a resolution of the singularities of $V$. The pull-back of $\widetilde{h}$ on $\widetilde{V}$ defines a smooth metric outside a divisor $F$ with negative holomorphic sectional curvature and locally bounded everywhere. Therefore from Theorem B in \cite{G18}, we obtain that $\widetilde{V}$ is of general type.
\end{proof}

\begin{proof}[Proof of Theorem B] Let $V \subset \widetilde{X}$, with $V \not\subset \mathbb B^+(L) \cup D \cup E$, and let $\widetilde{V}$ be a resolution of the singularities of $V$. Then the induced holomorphic map $\widetilde{V} \overset{j}{\longrightarrow} \widetilde{X}$ is generically immersive.

In Lemma \ref{lemcritere}, we can take $B = \frac{1}{n+1}$, $\gamma = \frac{2}{n+1}$. Then, choosing any positive constant $A < \frac{1}{n+1}$ will give a metric $\widetilde{h}$ such that $j^\ast \widetilde{h}$ is locally bounded everywhere on $T_{\widetilde{V}}$, has negative bisectional curvature and negative holomorphic sectional curvature bounded from above by $A - \frac{1}{n+1} < 0$. By \cite{cad16}, this implies that $\Omega_{\widetilde{V}}$ is big.  
\end{proof}

\section{Quotients of polydiscs}\label{disc}

In this section, we consider the case where $\Omega$ is the polydisk $\D^n$, and consider the same setting as above. Let $\Gamma \subset Aut(\D)^n$ be a discrete subgroup, and $X=\D^n/\Gamma$ a quotient smooth in codimension $1$, whose local isotropy groups have cardinal $m_i$.
 
We follow the same notations as above and denote $p : \Omega \longrightarrow X$ the canonical projection, $\widetilde{X} \overset{\pi}{\longrightarrow} \overline{X}$ a resolution with exceptional divisor $E_i$, $\widetilde{\Delta} = \sum_i \left(1 - \frac{1}{m_i} \right) E_i$, $E = \sum_i E_i$ and $\widetilde{D} = \pi^\ast D$.

\subsection{Symmetric differentials}
We will first establish a weak version of Theorem C giving a criteria for the bigness of the cotangent bundle.

\begin{theorem}\label{big1}
If the $\Q$-line bundle $\pi^*K_{\overline{X}}+\wtilde{D}-n(\wtilde {\Delta} + \widetilde{D})$ is big, then $\Omega_{\widetilde{X}}$ is big.
\end{theorem}

\begin{proof}
Let $\omega$ be a global section of $m(\pi^*K_{\overline{X}}+\wtilde{D})$. It gives a section of $K_{\D^n}^{\otimes m}$ invariant under $\Gamma$. Now we consider it as a section of $S^{mn} \Omega (\D^n)$ invariant under $\Gamma$. It gives a section of $S^{mn} \Omega (X^{reg})$. Therefore from the extension property \ref{prop:extension}, we obtain finally a section of $S^{mn}\Omega(\widetilde{X}, \wtilde{\Delta}+\widetilde{D})$. Starting with a section $\omega$ of a multiple of $$\pi^*K_{\overline{X}}+\wtilde{D}-n(\wtilde {\Delta}+\widetilde{D})-A,$$ where $A$ is an ample line bundle, one gets a section of $S^{mn}\Omega(\widetilde{X})\otimes A^{-1}$, which implies that $\Omega_{\widetilde{X}}$ is big.
\end{proof}

 \begin{example}
 Already in the case of surfaces, this statement provides interesting examples of surfaces with $c_1^2\leq c_2$ and big cotangent bundle (see \cite{GRR}).
 \end{example}

\subsection{Holomorphic foliations}

Let us suppose that $\Gamma$ is \emph{irreducible} in the following sense:
the restriction of each of the $n$ projections $p_j: Aut(\D)^n \to Aut(\D)$ to $\Gamma$ is injective. Remark that in this setting, singularities of quotients $X=\D^n/\Gamma$ are automatically cyclic quotient.

Let $\mathcal{F}_i$ be the holomorphic codimension-one foliation on $\widetilde{X}$ induced by $dz_i=0$ on $\D^n$. In the sequel, we will use properties of these foliations established in \cite{RT}, which we summarize in the following proposition.

\begin{prop}[\cite{RT}]\label{fol}
\begin{enumerate}
\item Leaves of the foliations $\mathcal{F}_i$, not contained in the exceptional part $E+\wtilde{D}$, do not contain algebraic varieties.
\item Leaves of the foliations $\mathcal{F}_i$, not contained in the exceptional part $E+\wtilde{D}$, are Brody hyperbolic.
\end{enumerate}
\end{prop}

\subsection{Proof of Theorem C}
Now, we can give the proof of the following statement.
 
 \begin{theorem}\label{polydisc}
Let $\Gamma \subset Aut(\D)^n$ be a discrete irreducible subgroup and denote $$L:=\pi^*K_{\overline{X}}+\wtilde{D}-n(\wtilde {\Delta} + \widetilde{D}).$$
Let $\B^+(L)$ be the augmented base locus of $L$ and $$Z:=\B^+(L) \cup \widetilde{D} \cup E.$$
Then all subvarieties $W \not \subset Z$ have big cotangent bundle.

In particular, if $L$ is big there is a proper surbvariety $Z \subsetneq \widetilde{X}$ such that all subvarieties $W \not \subset Z$ have big cotangent bundle.
 \end{theorem}
 
 \begin{proof}
 Let $W \subset \widetilde{X}$ be a subvariety.
 Following the same steps and notations as in the proof of theorem \ref{big1}, we see that the pull-backs of sections $\omega$ on $W$ may vanish in two cases: either $W \subset \B^+(L) \cup E \cup \widetilde{D}$ or $W$ is tangent to one of the codimension one holomorphic foliation $\mathcal{F}_i$ induced by $dz_i=0$ on $\D^n$. But proposition \ref{fol} says that this second alternative is not possible.
 \end{proof}
 
Finally, we observe that we obtain in this case another proof of Theorem A.
 
 \begin{theorem}\label{entire}
In the same setting as above, let $f:\C \to X$ be an entire curve. Then $f(\C) \subset Z:=\B^+(L) \cup \widetilde{D} \cup E$. In particular, if $L$ is big then $\Exc(X) \neq X$.
 \end{theorem}
 
\begin{proof}
Let $f:\C \to X$ be an entire curve. From the construction above, consider global symmetric differentials vanishing on an ample divisor induced by sections of $\pi^*K_{\overline{X}}-n(\wtilde {\Delta} + \widetilde{D})-A$. The classical vanishing theorem (see for example corollary 7.9 in \cite{De95}) implies that any entire curve $f: \C \to X$ satisfies $f^*\omega \equiv 0.$ Let us suppose that $f(\C)$ does not lie in $Z$ then $f$ has to be tangent to one of the holomorphic codimension-one foliations induced by $dz_i=0$. But this is impossible because proposition \ref{fol} says that these leaves are hyperbolic.
\end{proof}

\section{Hilbert modular varieties}\label{HMV}

Let $K$ be a totally real algebraic number field of degree $n>1$ over the rational number field $\Q$, and let $\mathcal{O}_K$ be the ring of integers in $K$. Then $\Gamma=\Gamma_K=SL_2(\mathcal{O}_K)$ acts on the product $\H^n$ of $n$ copies of the upper half plane $\H=\{z \in \C | \Im z > 0\}$: 
$$\left(\begin{array}{cc}\alpha & \beta \\ \gamma & \delta \end{array}\right).(z_1,\dots,z_n)=\left(\frac{\alpha^{(1)}z_1+\beta^{(1)}} {\gamma^{(1)}z_1+\delta^{(1)}},\cdots,
\frac{\alpha^{(n)}z_n+\beta^{(n)}}{\gamma^{(n)}z_n+\delta^{(n)}}
\right),$$
where $\alpha=\alpha^{(1)}, \alpha^{(2)}, \dots, \alpha^{(n)}$ denote the conjugates of $\alpha \in K$.

A holomorphic function $f$ on $\H^n$ is called a Hilbert modular form of weight $k$ if
$$f(Mz)=N(\gamma z+\delta)^kf(z)$$
for all $M=\left(\begin{array}{cc}\alpha & \beta \\ \gamma & \delta \end{array}\right) \in \Gamma,$ where $N(z)=\prod_{i=1}^n z_i.$

Hibert modular forms are classically interpreted in terms of differential forms: if $\omega=dz_1\wedge \dots \wedge dz_n$ and $f$ is a Hilbert modular form of weight $k$ then $f\omega^{\otimes k}$ gives an invariant holomorphic top-differential forms which descends on (the smooth part of) $\H^n/\Gamma.$

The observation already used in the previous sections is that one can also look at Hilbert modular forms as symmetric differential forms. Indeed, in the above notations, $f(dz_1\dots dz_n)^k$ is also invariant under $\Gamma$ and therefore provides a symmetric differential on (the smooth part of) $\H^n/\Gamma.$ 

Recall that there is a natural compactification $Y:=\overline{\H^n/\Gamma}$ adding finitely many cusps. Then one can take a projective resolution $X \to Y$.

Now we can apply the results of the previous section.

First, a corollary of Theorem \ref{big1} gives the following result.

\begin{theorem}\label{bigHMV}
Let $n \geq 2$. Then except finitely many possible exceptions Hilbert modular varieties have a big cotangent bundle.
\end{theorem}

\begin{proof}
Let $E=E_e+E_c$ be the exceptional divisor where $E_e=\sum_i E_e^i$ corresponds to the resolution of elliptic points and $E_c$ corresponds to the resolution of cusps. Theorem \ref{big1} tells us that there are constants $\alpha_i$ depending only on the order of the stabilizer of elliptic fixed points such that if $K_X+E-\sum_i \alpha_i E_e^i - nE_c$ is big then $\Omega_X$ is big.

Therefore we are reduced to prove that $K_X+E-\sum_i \alpha_i E_e^i - nE_c$ is big except finitely many possible exceptions. Let $S_k^m$ denote the space of Hilbert modular forms of weight $k$ and vanishing order at least $m$ over cusps. Sections of $$k(K_X+E-\sum_i \alpha_i E_e^i - nE_c)$$ corresponds to modular forms. We have to show the maximal growth of the space of corresponding modular forms.
So we have to prove that one can produce more sections than the number of conditions imposed by the vanishing along the exceptional components.
We shall use the following result of \cite{Tsu85} (Sect. 4)
\begin{equation}\label{RR}
\dim S_k^{\nu k}(\Gamma_K) \geq (2^{-2n+1}\pi^{-2n} d_K^{3/2} \zeta_K(2)-2^{n-1}\nu^nn^{-n}d_K^{1/2}hR)k^n+O(k^{n-1})
\end{equation}
for even $k \geq 0$, where $h, d_K, R, \zeta_K$ denote the class number of $K$, the absolute value of the discriminant, the positive regulator and the zeta function of $K$.
In particular, there is a modular form $F$ with $\ord (f)/\weight(f) \geq \nu$, if
\begin{equation}\label{weight}
\nu < 2^{-3}\pi^{-2}n\left( \frac{4d_K\zeta_K(2)}{hR}\right)^{1/n}.
\end{equation}
If we fix $n$, then $\zeta_K(2)$ has a positive lower bound independent of $K$. Since $hR \sim d_K^{1/2}$ by the Brauer-Siegel Theorem, for any constant $C$ there are only a finite number of $K$ such that the right hand side of (\ref{weight}) is smaller than $C$.
Therefore we obtain sections with the requested order of vanishing along $E_c$.

Let us now deal with the elliptic points.
Prestel \cite{Prestel} has obtained precise formula on the number of elliptic points of the Hilbert modular group.
In particular, one can deduce (see section 6.5 of \cite{RT} for details) that for fixed $n$, there are only finitely many different type of elliptic points and the number of equivalence classes of elliptic fixed points is $O(d_K^{\frac{1}{2}+\epsilon})$ for every $\epsilon>0$.
This immediately gives the maximal growth of the space of modular forms satisfying the vanishing conditions with finitely many possible exceptions.
\end{proof}

As a consequence, we recover the main result of \cite{Tsu85}
\begin{corollary}
Let $n \geq 2$. Then except finitely many possible exceptions Hilbert modular varieties are of general type.
\end{corollary}
\begin{proof}
This is an immediate application of \cite{CP15} who prove that if the cotangent bundle is big then the canonical bundle is big.
\end{proof}

Now, we can give the proof of the two statements announced in Theorem D of the introduction as corollaries of Theorems \ref{polydisc} and \ref{entire}.

\begin{theorem}
Let $n \geq 2$. Then, except finitely many possible exceptions, Hilbert modular varieties satisfy $\Exc(X) \neq X$.
\end{theorem}
\begin{proof}
Let $X$ be a Hilbert modular variety. The proof of Theorem \ref{bigHMV} tells that $L:=K_X+E-\sum_i \alpha_i E_e^i - nE_c$ is big except finitely many possible exceptions. Then Theorem \ref{entire} tells us that $\Exc(X) \subset \B^{+}(L)\cup E.$
\end{proof}

Finally, we obtain the second statement.

\begin{theorem}
Let $n \geq 2$. Then except finitely many possible exceptions Hilbert modular varieties contain a proper subvariety $Z$ such that all subvarieties not contained in $Z$ have big cotangent bundle and are of general type.
\end{theorem}

\begin{proof}
Let $X$ be a Hilbert modular variety such that $L:=K_X+E-\sum_i \alpha_i E_e^i - nE_c$ is big. Then define $Z:=\B^{+}(L)\cup E$. Let $Y \subset X$ be a subvariety not contained in $Z$. Theorem \ref{polydisc} gives that all subvarieties not contained in $Z$ have big cotangent bundle and are of general type.
\end{proof}

\bibliography{bibliography}{}

\end{document}